\documentclass[a4paper,twoside,11pt]{amsart}      
\usepackage{amsmath,amssymb}
\usepackage{amsthm}
\usepackage{graphics}                 
\usepackage{color}                    
\usepackage{hyperref}                 
\usepackage{graphicx,eucal}
\usepackage[all]{xy}
\usepackage{soul}
\usepackage{geometry}
\usepackage{changebar}

\newtheorem{theorem}{Theorem}[section]
\newtheorem{proposition}[theorem]{Proposition}
\newtheorem{corollary}[theorem]{Corollary}
\newtheorem{lemma}[theorem]{Lemma}

\newtheorem{example}[theorem]{Example}

\theoremstyle{definition}

\newtheorem{remark}[theorem]{Remark}

\newcommand{\Z}{\mathbb{Z}}
\newcommand{\Q}{\mathbb{Q}}
\newcommand{\R}{\mathbb{R}}
\newcommand{\C}{\mathbb{C}}

\newcommand{\OO}{\mathcal{O}}
\newcommand{\F}{\mathcal{F}}
\newcommand{\G}{\mathcal{G}}
\newcommand{\T}{T}

\newcommand{\lin}{\mathrm{lin}}
\newcommand{\Diff}{\mathrm{Diff}}
\newcommand{\Tang}{\mathrm{Tang}}

\newcommand{\Dom}{D}
\newcommand{\Disc}{\Delta}

\newcommand{\Mon}{\mathrm{Mon}}

\date{\today}
\title[Mattei-Moussu's Theorem]{Pairs of foliations and Mattei-Moussu's Theorem.}

\author[A. A. Diaw]{Adjaratou Arame Diaw}
\address{Univ Rennes, CNRS, IRMAR - UMR 6625, F-35000 Rennes, France}
\email{aramdiaw2@gmail.com}

\author[F. Loray]{Frank Loray}
\address{Univ Rennes, CNRS, IRMAR - UMR 6625, F-35000 Rennes, France}
\email{frank.loray@univ-rennes1.fr}

\dedicatory{\`A J.-F. Mattei, J. Martinet, J.-P. Ramis et R. Moussu.}


\begin{document}

\begin{abstract}
We prove a reduction of singularities for pairs of foliations by blowing-up, and then investigate
the analytic classification of the reduced models. Those reduced pairs of regular foliations are 
well understood. The case of a regular and a singular foliation is dealt with Mattei-Moussu's Theorem
for which we provide a new proof, avoiding Gronwall's inequality. We end-up announcing results recently
obtained by the first author in the case of a pair of reduced foliations sharing the same separatrices.
\end{abstract}
  
\thanks{We warmly thank Jorge Vit\'orio Pereira and Fr\'ed\'eric Touzet for useful discussions.
This work is supported by ANR-16-CE40-0008 ``Foliage'' grant,
CAPES-COFECUB Ma 932/19 project, Universit\'e de Rennes 1 and Centre Henri Lebesgue CHL}
\maketitle

\section{Introduction}

A singular holomorphic foliation by curves on a surface is locally given by a holomorphic vector field
with isolated zeroes. Outside of the zero set, the complex integral curves of the vector field $v$ define 
a regular holomorphic foliation by curves $\F$. Another vector field will define the same foliation
if, and only if, it takes the form $f\cdot v$ for a non vanishing holomorphic function $f$.

In \cite{MatteiMoussu}, Mattei and Moussu provide a topological characterization of singular points
of holomorphic foliations that admit a non constant holomorphic first integral, i.e. such that the leaves 
are locally defined as the level curves of a holomorphic function $h$. For this, they use the reduction 
of singularities by blow-up, and they provide a careful study of the reduced singular points. In particular,
they prove that the saddle singular points are determined by their eigenvalues and holonomy map
(see Theorem \ref{thm:MM}, and also \cite{EI}). 
The proof of this famous result is in two parts. They consider two saddles with the same eigenvalues
$\{\lambda_1,\lambda_2\}$ and assume that the holomomies of the separatrices associated to $\lambda_1$ 
say are conjugated.
First they use this conjugacy to construct a conjugacy of the foliations between neighborhoods of annuli
contained into the separatrices; this is done by a simple and standard geometric argument. In a second step,
they prove that the conjugacy extends on a neighborhood of the singular point minus the second separatrix,
by controlling the boundedness by means of Gronwall's Inequality, and extends holomorphically 
along the separatrix by Riemann Extension Theorem. This second part of the proof is delicate and, 
in general, not written in full details. 
Here, we provide in Section \ref{Sec:MM} a complete and elementary proof avoiding Gronwall's Inequality. 
In fact, we compare the analytic conjugacy on the annulus with the formal conjugacy which is transversely formal 
on the disc:  they differ by a transversely formal symmetry of the foliation on the annulus. 
It suffices to prove that all such symmetry extend as a transversely formal symmetry on the disc, 
and this allow to conclude the analytic extension of the conjugacy. Our proof remains valid 
in the saddle-node case when the central manifold (invariant curve tangent to the zero-eigendirection)
is analytic (not only formal) although Gronwall's Inequality cannot be used to conclude in that case.
We hope that this proof can be useful in other situations, like for higher dimensional saddles,
for instance weakening assumptions of Reis' result in  \cite{Reis}.
At least, this approach has been recently used by the first author to classify some natural pairs of singular foliations.
Let us explain.

Mattei-Moussu's Theorem provide fibered conjugacy, that is to say, it can be thought as classification
of pairs $(\F_1,\F_2)$ of a saddle $\F_2$ with a regular foliation $\F_1$ such that one of the separatrices of $\F_2$
is a leaf of $\F_1$. We start a study of pairs of singular foliations by proving a reduction of singularities.
In Theorem \ref{thm:ReductionPair}, we prove that, given a pair of singular foliations on a complex surface $M$,
we can construct a proper map $\pi:\tilde M\to M$ obtained by a finite sequence of blowing-up
such that the singular points of the lifted pair $\pi^*(\F_1,\F_2)$ fits with a list of simple models.
Reduced pairs of regular foliations can be easily handled and are well-known, let us mention 
\cite{OlivierBras} for a recent contribution.
Mattei-Moussu's Theorem and some results of Martinet and Ramis allows us to deal with the analytic
classification of those models where only one of the two foliations is regular. 
In Section \ref{Sec:Pairs}, we state a recent contribution of the first author to the case of two singular foliations:
reduced pairs consist of pairs of reduced singular foliations sharing the same invariant curves.
In \cite{TheseArame,PapierPairFol}, the first author provides a complete analytic classification
in the case the two foliations have a reduced tangency divisor (i.e. without multiplicity)
along the invariant curves (see Theorem \ref{thm:Pair}). All details will be published in the 
forthcoming paper \cite{PapierPairFol}. In the study of reduced pairs, only remains the case
of two reduced singular foliations with higher tangency along the invariant curves. This seems
feasible, but more technical and left to the future.

\section{Singular holomorphic foliations in dimension 2}

A singular holomorphic foliation by curves $\F$ on a complex surface $M$ is defined 
by coherent analytic subsheaf $T_\F \subset TM$ of rank 1 such that the quotient sheaf $TM/ T_\F$
is torsion free. The locus where $TM/ T_\F$ is not locally free is the singular locus, a discrete set.
Locally, sections of $T_\F$ take the form $f\cdot v$ where $v$ is a holomorphic vector field 
with isolated zeroes, and $f$ is any holomorphic function. Any other generating vector field takes the form 
$f\cdot v$ with $f$ non vanishing. Equivalently, one can locally define $T_\F$ as the kernel of a 
holomorphic $1$-form $\omega$ with isolated zeroes, which is also defined up to a non vanishing factor, 
and globally by the conormal bundle $N_\F^*$,
a rank 1 coherent subsheaf of $T_\F^*$ with torsion-free cokernel (see \cite[Chapter 2]{Brunella}).

Let $v$ be a germ of holomorphic vector field at $(\mathbb C^2,0)$:
$$v=a(x,y)\partial_x+b(x,y)\partial_y,\ \ \ a,b\in\C\{x,y\}$$
and we assume $0$ is an isolated singular point of $v$, i.e. the common zero of $a$ and $b$.
The linear part of $v$ is defined as the linear part of the map $V:(x,y)\mapsto(a(x,y),b(x,y))$, 
i.e. its differential $D_0V$ at $0$; we denote by $\lin(v)$ this linear map. If $\Phi\in\Diff(\C,0)$ is a 
coordinate change, then we have 
$$\lin\left(\Phi^*v)\right)=D_0\Phi^{-1}\cdot \lin(v) \cdot D_0\Phi.$$
Therefore, the eigenvalues $\{\lambda_1,\lambda_2\}$
of $\lin(v)$ are invariant under change of coordinates. If we are interested in the foliation $\F$
defined by $v$, then we note that 
$$\lin(f\cdot v)=f(0)\cdot\lin(v)$$
so that only the ratio $\lambda=\lambda_2/\lambda_1$ (when $\lambda_1\not=0$ say)
is well defined by the foliation.
When $\lambda\not\in\R_{\le0}$, then Poincar\'e and Dulac proved that we can reduce 
the foliation to the unique normal form 
\begin{equation}\label{eq:PoincareDulac}
x\partial_x+\lambda y\partial_y,\ \ \ \text{or}\ \ \ x\partial_x+(ny+x^n)\partial_y
\end{equation}
by an analytic change of coordinate; the latter case occur when $\lambda$ or $\frac{1}{\lambda}=n\in\Z_{>0}$.
When $\lambda<0$, we are in the saddle case. Briot and Bouquet proved that there are two invariant curves, 
one in each eigendirection, so that after straigthening them on coordinates axis, we get a preliminary normalization
\begin{equation}\label{eq:BriotBouquet}
x\partial_x+(\lambda+f(x,y)) y\partial_y
\end{equation}
where $f$ is a function vanishing at $0$. Poincar\'e and Dulac proved that we can reduce 
the foliation to the unique normal form by a formal change of coordinate:
\begin{equation}\label{eq:PoincareDulacSaddle}
x\partial_x+\lambda y\partial_y,\ \ \ \text{or}\ \ \ x\partial_x+\left(-\frac{p}{q}+u^k+\alpha u^{2k}\right)y\partial_y\ \ \ 
\text{with}\ \left\{\begin{matrix}
p,q,k\in\Z_{>0}\\ u=x^py^q,\\ \alpha\in\C
\end{matrix}\right.
\end{equation}
The latter case, called resonant saddle, occurs for generic $f$ when $\lambda\in\Q_{<0}$.
We will provide below a proof of this formal reduction.  It is known however that this formal normalization 
is divergent in general, and the analytic classification fails to be finite dimensional in this case. 
In fact, the classification of saddles up to analytic conjugacy
is equivalent to the classification of germs of diffeomorphisms $\Diff(\C,0)$ up to analytic conjugacy.
Before explaining this, let us just end our review of non degenerate singular points by the saddle-node case
$\lambda=0$ which admits a normal form similar to (\ref{eq:PoincareDulacSaddle}) with $(p,q)=0$ and $u=y$:
\begin{equation}\label{eq:PoincareDulacSaddleNode}
x\partial_x+\left(y^{k+1}+\alpha y^{2k+1}\right)\partial_y\ \ \ 
\text{with}\ \left\{\begin{matrix}
k\in\Z_{>0}\\  \alpha\in\C
\end{matrix}\right.
\end{equation}
In order to explain this, let us recall the construction of the holonomy.

\subsection{Holonomy}\label{sec:Hol}

Fix a disc $\Disc=\{(x,0)\ ;\ \vert x\vert<r\}\subset\C^2$ such that the foliation is of the form (\ref{eq:BriotBouquet}),
with $f$ analytic at the neighborhood of $\Disc$, $\lambda+f$ non vanishing on $\Disc$. 
Denote by $\Disc^*$ the punctured disc, where we have deleted 
the singular point of $\F$. Fix a point $p\in\Disc^*$ and choose a minimal first $H(x,y)\in\OO(\C^2,p)$:
locally at $p$, the $1$-form $dH$ is non vanishing and its kernel defines the foliation, i.e.
$dH(v)=v\cdot H=0$. Consider a loop $\gamma:[0,1]\to\Disc^*$  based at $p=\gamma(0)=\gamma(1)$,
and consider the analytic continuation of $H$ along $\gamma$. 

We claim that such a continuation exists because $\gamma$ is contained in a regular leaf $\Disc^*$ of the foliation. 
Indeed, one can cover $\gamma$ by small balls $U_i$
equipped with a local minimal first integral $H_i$ and choose finitely many of them by compacity:
we enumerate such that $\gamma$ successively intersect $U_0,U_1,\ldots,U_n$. 
Denote $V_i=U_i\cap\Disc^*$.
Near overlappings $V_{i}\cap V_{i+1}$, we have $H_{i}=\varphi_{i,i+1}\circ H_{i+1}$
for some $\varphi_{i,i+1}\in\Diff(\C,0)$ because they both define minimal first integrals;
therefore, $\varphi_{i,i+1}\circ H_{i+1}$ defines an analytic extension of $H_i$ 
near $V_{i+1}$. Starting from $(U_0,H_0)=(U,H)$, the analytic continuation $H^\gamma$ 
of $H$ along $\gamma$ is given by:
$$H=\underbrace{H_0}_{V_0}=\underbrace{\varphi_{0,1}\circ H_1}_{V_1}=\underbrace{\varphi_{0,1}\circ\varphi_{1,2}\circ H_2}_{V_2}=\cdots=\underbrace{\varphi_{0,1}\circ\varphi_{1,2}\circ\cdots\circ\varphi_{n,0}\circ H_0}_{\text{back to}\ V_0}=:H^\gamma.$$
By construction, $H^\gamma$ is a minimal first integral of $\F$ near $V=U\cap\Disc^*$, and we have 
$$H^\gamma=\varphi^\gamma\circ H$$
with $\varphi^\gamma\in\Diff(\C,0)$. It is easy to check that $\varphi^\gamma$ only depend on the 
homotopy type of $\gamma$; it is called the monodromy of $H$ along $\gamma$. Moreover, we get a group morphism
$$\Mon(H)\ :\ \pi_1(\Disc^*,p)\to\Diff(\C,0)\ ;\ \gamma\mapsto \varphi^\gamma.$$
Finally, if we start with another minimal first integral $\tilde H=\varphi\circ H$, then the monodromy
is changed by 
$$\tilde H^\gamma=\varphi\circ H^\gamma=\varphi\circ\varphi^\gamma\circ H
=\underbrace{\varphi\circ\varphi^\gamma\circ\varphi^{-1}}_{\tilde\varphi^\gamma}\circ\tilde H.$$
We call holonomy of $\F$ along $\gamma$ the class of $\varphi^\gamma$ up to conjugacy
in $\Diff(\C,0)$ (or a representative by abuse of notation); in the sequel, 
we will call holonomy of $\F$ along the leaf $\Disc^*$
the holonomy of $\F$ along a loop of index one, i.e. of the form $\gamma(t)=(x_0e^{2i\pi t},0)$.
One can check that the holonomy of $\F$ of the form (\ref{eq:BriotBouquet}) along $\Disc^*$
is of the form 
\begin{equation}\label{eq:LinearPartHol}
\varphi(z)=e^{-2i\pi\lambda}z+o(z).
\end{equation}

\begin{example}\label{ex:FormalHolonomy}
The holonomy of the linear model in \ref{eq:PoincareDulacSaddle} is the monodromy of 
the local first integral $H(x,y)=x^{-\lambda}y$ (after choosing a local determination of $x^{-\lambda}=\exp(-\lambda\log(x))$),
namely $\varphi(z)= e^{-2i\pi\lambda} z$. Equivalently, one can integrate the vector field $v=x\partial_x+\lambda y\partial_y$
and deduce a global symmetry $\phi(x,y)=\exp(-2i\pi v)=(x,e^{-2i\pi\lambda}y)$ of the foliation;
and then restrict to a transversal $x=x_0$. 

The holonomy of the non-linear model in \ref{eq:PoincareDulacSaddle} is $\varphi(z)=e^{2i\pi\frac{p}{q}}\exp(v)$ where $v$ is the holomorphic vector field 
$$v=-2i\pi\left(z^{kq+1}+\alpha z^{2kq+1}\right)\partial_z.$$
Indeed, after introducing variables $x={\tilde x}^q$ and $z={\tilde x}^py$, we get $u=z^q$ and the foliation is defined by
$$\tilde\F\ :\ \tilde x\partial_{\tilde x}+q(z^{kq+1}+\alpha z^{2kq+1})\partial_z$$
which is the formal normal form for a saddle-node singular point. The holonomy
of $\tilde\F$ along $\tilde\Disc^*=\{(\tilde x,0)\ ;\ 0<\vert\tilde x\vert<\tilde r\}$ is the $q^{\text{th}}$ iterate 
of the holonomy of $\F$ along $\Disc^*$, namely 
$$\tilde\varphi=\varphi^{\circ q}=\exp(q\cdot v).$$
We conclude by noticing that any $q^{\text{th}}$ root takes the form $\varphi=a\cdot\exp(v)$, and 
we must have the linear part $a=e^{-2i\pi\lambda}=e^{2i\pi\frac{p}{q}}$
(compare \cite{MartinetRamisENS} and \cite{MartinetRamis}). This latter computation is also valid 
for the saddle-node case (\ref{eq:PoincareDulacSaddleNode}) by setting $(p,q)=(0,1)$ and $u=y$.
\end{example}

Mattei-Moussu's Theorem tells us that any two saddles of the form (\ref{eq:BriotBouquet})
are conjugated by an analytic diffeomorphism preserving the two separatrices (i.e. not permuting the coordinate axis)
 if, and only if, they share the same invariant $\lambda$, and the corresponding holonomies 
 along $\Disc^*$ are analytically conjugated. On the other hand, P\'erez-Marco and Yoccoz proved
 in \cite{PMY} that any diffeomorphism germ $\varphi(z)=e^{-2i\pi\lambda}z+o(z)$ is the holonomy 
 of a saddle of the form (\ref{eq:BriotBouquet}). Therefore, for a fixed $\lambda<0$, analytic classification
 of saddles and one-dimensional diffeomorphisms are equivalent. 
 A result of Siegel shows that for generic (with respect to Lebesgue measure) $\lambda$,
 any diffeomorphism, and therefore any saddle is analytically linearizable. But for special
 ``diophantine'' $\lambda$, the moduli space is infinite dimensional as shown by the works
 of Yoccoz in the irrational case, and \'Ecalle-Malgrange-Voronin in the rational case. 
 Martinet and Ramis gave a complete analytic classification of resonant saddles 
(i.e. non linear case of  (\ref{eq:PoincareDulacSaddle})) in \cite{MartinetRamisENS},
and of saddle-nodes $\lambda=0$ in \cite{MartinetRamis}.

\subsection{Normalization in the transversely formal setting}
Let $\Dom\subset\C$ be a domain.  A transversely formal function on $\Dom\times(\C,0)\ni(x,y)$
is a power series 
$$f=\sum_{n\ge0}a_n(x)y^n\in\OO(\Dom)[[y]]$$
i.e. where all $a_n$ are analytic on $\Dom$, and nothing is asked about convergence in the $y$-variable.
A transversely formal  diffeomorphism is a ``map'' $\Phi(x,y)=(x+yf(x,y),yg(x,y))$ where $f,g\in\OO(\Dom)[[y]]$
and $g(x,0)$ does not vanish on $\Dom$. We denote by $\Diff(\Dom\times\widehat{(\C,0)})$
the group of transversely formal diffeomorphisms.

\begin{proposition}\label{prop:PoincareDulacTransForm}
Let $\Disc=\{x\ ;\ \vert x\vert<r\}$ be a disc and let
$\F$ be the foliation defined by (\ref{eq:BriotBouquet}) with $f$ transversely formal on $\Disc\times(\C,0)\ni(x,y)$ 
(for instance, analytic on the neighborhood of $\Disc\times\{0\}$). 
Then there exists a transversely formal diffeomorphism
$\Phi\in\Diff(\Dom\times\widehat{(\C,0)})$  of the form
$$\Phi(x,y)=(x,yg(x,y))$$
such that $\Phi^*\F$ is defined by (\ref{eq:PoincareDulacSaddle}).
\end{proposition}

\begin{remark}The resonant case of (\ref{eq:PoincareDulac}) does not occur since 
we have imposed in expression (\ref{eq:BriotBouquet}) that $\F$ has two invariant curves.
Also, in the saddle-node case $\lambda=0$, the central manifold, tangent to the $0$-eigendirection, 
is convergent in expression (\ref{eq:BriotBouquet}).
\end{remark}

\begin{proof}It is more convenient to define $\F$ as the kernel of the $1$-form
$$\omega=xdy-(\lambda+f)ydx,\ \ \ \text{or better}\ \frac{dy}{y}-(\lambda+f)\frac{dx}{x}$$
as coordinate change will be easier to handle on $1$-forms. Our strategy is to 
make successive changes of $y$-coordinate to kill (or reduce) step-by-step the 
coefficients $a_n(x)$ in $f=\sum_na_n(x)y^n$. Let us first consider a linear
change $y\mapsto\varphi_0(x)y$. Then $\Phi^*\omega$ writes
$$\frac{dy}{y}-\left(\lambda+\left(a_0(x)-\frac{x\varphi_0'(x)}{\varphi_0(x)}\right)y+o(y)\right)\frac{dx}{x};$$
as $a_0$ vanishes at $x=0$, we can integrate and find $\varphi_0=\exp(\int\frac{a_0}{x})\in\OO^*(\Disc)$.
Now we can assume $a_0\equiv0$, i.e. $f(x,0)\equiv0$,  and all further changes of coordinate will be tangent 
to the identity, i.e. with $g(x,0)\equiv1$.
Consider now a change of the form $y\mapsto y(1+\varphi_n(x)y^n)$, $n>0$. Then $\Phi^*\omega$ writes
after normalization (i.e. multiplication by a non vanishing analytic function)
$$\frac{dy}{y}-\left(\lambda+f-\left(n\lambda\varphi_n+x\varphi_n'\right)y^n+o(y^n)\right)\frac{dx}{x}.$$
Since linear operator 
$$ \OO(\Disc)\to\OO(\Disc)\ ;\ \varphi\mapsto n\lambda\varphi+x\varphi'$$
writes
$$\sum_{m\ge0}b_mx^m\ \mapsto\ \sum_{m\ge0}(n\lambda+m)b_mx^m$$
we deduce, when $\lambda\not\in\Q_{\le0}$, that the operator is onto, so that we can find 
at each step a $\varphi_n$ killing the coefficient $a_n$ of $f$. The composition 
$$\Phi=\cdots \circ\Phi_n\circ\cdots\circ\Phi_2\circ\Phi_1\circ\Phi_0$$
converges in $\Diff(\Dom\times\widehat{(\C,0)})$ providing a linearization of $\F$, i.e. $\Phi^*\omega$
is colinear to $xdy-\lambda ydx$. Indeed, for each $n>0$, the jet of order $n$ of $\Phi$ in $y$-variable
is determined by $\Phi_n\circ\cdots\circ\Phi_2\circ\Phi_1\circ\Phi_0$ since all other terms in the composition
are tangent to the identity up to order $n$. 
Assume now that $\lambda\in\Q_{<0}$, i.e. $\lambda=-\frac{p}{q}$ with $p,q\in\Z_{>0}$ relatively prime.
Then we see that we can kill successively all the terms expect powers of $u=x^py^q$ in $f$.
Therefore, we are led by a transversely formal diffeomorphism to the preliminary normal form 
$$(p-f(u))\frac{dx}{x}+q\frac{dy}{y}=f(u)\left(\frac{du}{uf(u)}-\frac{dx}{x}\right)$$
where $f$ is now a formal power series, with $f(0)=0$. One easily check that the freedom 
in this normalization is up to a change of the form $\Phi(x,y)=(x,y\psi(u))$, $\psi(0)\not=0$;
moreover, this induces a change
$$u\mapsto \varphi(u)=u(\psi(x))^q,\ \ \ \text{and therefore}\ \ \ \Phi^*\left(\frac{du}{uf(u)}-\frac{dx}{x}\right)=\varphi^*\left(\frac{du}{uf(u)}\right)-\frac{dx}{x}$$
and we are led to a normalization of a (formal) meromorphic $1$-form in one variable.
It is well known (see proof of \cite[Proposition 2.1]{MartinetRamisENS} or \cite[Proposition 1.1.3]{Lecons}) that there exists a unique change of coordinate $\varphi$ tangent to the identity such that
$$\varphi^*\left(\frac{du}{uf(u)}\right)=\frac{du}{u^{k+1}}+\alpha\frac{du}{u}$$
which is analytic (resp. formal) if $f$ is analytic (resp. formal). The only invariants are the pole order $k+1$ (here, $k>0$ is the vanishing order of $f$), and the residue $\alpha$. 
The last normalization is therefore given by the choice of a determination of  $\psi(u)=\left(\frac{\varphi(u)}{u}\right)^{1/q}$.
In the normal form (\ref{eq:PoincareDulacSaddle}), we have choosen another normalization,
namely $\frac{du}{qu^{k+1}(1+\alpha u^k)}$ where the residue is $-\frac{\alpha}{q}$.
Finally, in the saddle-node case $\lambda=0$, we note that everything works the same with $u=y$ (i.e. $p=0$ and $q=1$).
\end{proof}

We deduce a formal version of Mattei-Moussu's Theorem:

\begin{corollary}\label{Cor:FormalMatteiMoussu}
Let $\F_1$ and $\F_2$ be two foliations of the form
$$\F_i=\ker(\omega_i),\ \ \ \omega_i=xdy-(\lambda+o(y))ydx,\ \ \ \lambda\in\C$$
and assume they are both analytic at the neighborhood of the disc $\Disc$.
Then, there exists a transversally formal diffeomorphism $\hat\Phi(x,y)=(x,y+o(y))$ along $\Disc$
conjugating the foliations $\hat\Phi^*\F_2\wedge\F_1$,
i.e. $\hat\Phi^*\omega_2\wedge\omega_1=0$, if and only if, 
the respective holonomies along the punctured disc $\Disc^*$ are formally conjugated.
\end{corollary}

\begin{proof}If we have a transversally formal conjugacy between foliations along even the punctured disc
is enough to conclude that the holonomies are conjugated. For the converse, we can assume that 
$\F_1$ and $\F_2$ are into normal form (\ref{eq:PoincareDulacSaddle}), with the same $\lambda$
(resp. (\ref{eq:PoincareDulacSaddleNode}) when $\lambda=0$.
Only the resonant saddle case $\lambda=-\frac{p}{q}$ and saddle-node case $\lambda=0$
need some arguments, as we have to distinguish between the linear case, and the several possible non-linear cases
of (\ref{eq:PoincareDulacSaddle}), i.e how to retrieve $k$ and $\alpha$ from the holonomy. 
But, as shown in Example \ref{ex:FormalHolonomy}, the holonomy
of the non-linear saddle (\ref{eq:PoincareDulacSaddle}) or saddle-node (\ref{eq:PoincareDulacSaddleNode}) takes the form
$$\varphi^{\circ q}(z)=\exp\left(  \left( z^{kq+1}-\frac{\alpha}{2i\pi q}z^{2kq+1}\right) \partial_z\right) 
=z+z^{kq+1}+\left(\frac{kq+1}{2}-\frac{\alpha}{2i\pi q}\right)z^{2kq+1}+o\left( z^{2kq+1}\right) .$$
It is well-known (see \cite[p. 581]{MartinetRamisENS} or \cite[Section1.3]{Lecons}) that the formal conjugacy
class of the tangent-to-identity diffeomorphism $\varphi^{\circ q}$ is characterized by two invariants, namely
the maximal contact $kq$ to the identity, and the coefficient of the monomial $z^{2kq+1}$ once we have
killed intermediate coefficients; the first one gives us $k$, and we deduce $\alpha$ from the latter one.
\end{proof}

\section{Mattei-Moussu's Theorem}\label{Sec:MM}

\begin{theorem}[Mattei-Moussu]\label{thm:MM}
Let $\F_1$ and $\F_2$ be two foliations of the form
$$\F_i=\ker(\omega_i),\ \ \ \omega_i=xdy-(\lambda+o(y))ydx,\ \ \ \lambda\in\R_{<0}.$$
Assume that they are well defined near $\Disc^*=\{(x,0)\ ;\ 0<\vert x\vert<r\}$ and have same holonomy along this leaf.
Then, there exists an analytic diffeomorphism $\Phi(x,y)=(x,y+o(y))$ conjugating the foliations: $\Phi^*\F_2\wedge\F_1$,
i.e. $\Phi^*\omega_2\wedge\omega_1=0$.
\end{theorem}

\begin{proof}The first part of the proof follows the paper \cite{MatteiMoussu}.
Let $p_0\in\Disc^*$ and $H_1,H_2$ be local first integrals of $\F_1,\F_2$ respectively at $p_0$.
Assume that $H_1$ and $H_2$ have same monodromy $\varphi=\varphi^\gamma\in\Diff(\C,0)$
along a generating loop $\gamma$ of $\pi_1(\Disc^*)$.
Notice that the local diffeomorphism $\Phi_i(x,y)=(x,H_i(x,y))$ at $p_0$ is conjugating 
$\F_i$ to the horizontal foliation $\{y=\text{constant}\}$, i.e. $\Phi_i^*dy\wedge\omega_i=0$. 
Therefore, the local diffeomorphism $\Phi:=(\Phi_2)^{-1}\circ\Phi_1$ is conjugating $\F_1$ to $\F_2$
as in the statement, near $p_0$. By analytic continuation of $H_1,H_2$ along paths in $\Disc^*$
(see Section \ref{sec:Hol}) one can deduce the analytic continuation of $\Phi$ by the same formula.
Since $H_1$ and $H_2$ have the same monodromy $\varphi$, we see that $\Phi$ is uniform: 
after analytic continuation along the generating loop $\gamma$, we get
$$\Phi^\gamma=(\Phi_2^\gamma)^{-1}\circ\Phi_1^\gamma=(\phi\circ\Phi_2)^{-1}\circ(\phi\circ\Phi_1)
=(\Phi_2)^{-1}\circ\phi^{-1}\circ\phi\circ\Phi_1=(\Phi_2)^{-1}\circ\Phi_1=\Phi$$
where $\phi(x,y)=(x,\varphi)$. We therefore obtain an analytic diffeomorphism $\Phi:U_1\to U_2$ 
between neighborhoods $U_i$ of $\Disc^*$ conjugating the restrictions $\F_1$ to $\F_2$.

In the second part of the proof, we want to extend $\Phi$ analytically on the neighborhood of $0$.
Our idea is to compare $\Phi$ with the transversely formal diffeomorphism $\hat\Phi$ along $\Disc$
provided by Corollary \ref{Cor:FormalMatteiMoussu}. As they are both conjugating $\F_1$ to $\F_2$,
$$\hat\phi:=\hat\Phi^{-1}\circ\Phi$$
defines a transversely formal diffeomorphism along $\Disc^*$ that preserves the foliation $\F_1$.
We are going to show, in Lemma \ref{lem:SymmetryMM}, that this forces $\hat\phi$ to extends as 
a transversely formal diffeomorphism on the whole disc $\Disc$. Therefore, the diffeomorphism
$\hat\Phi\circ\hat\phi=\Phi$ is both analytic near $\Disc^*$ and transversely formal on $\Disc$.
Equivalently, the Laurent power series decomposition of its coefficients are convergent and with only positive
exponents: they define holomorphic germs at $0$, therefore extending $\Phi$.
\end{proof}

\begin{lemma}\label{lem:SymmetryMM}
Let $\F$ be a foliation of the form
$$\F=\ker(\omega),\ \ \ \omega=xdy-\left( \lambda+o(y)\right) ydx,\ \ \ \lambda\in\C\setminus\Q_{>0}$$
analytic near the disc $\Disc=\{(x,0)\ ;\ \vert x\vert<r\}$.
Then, any transversely formal diffeomorphism $\hat\Phi(x,y)=(x,\hat\phi(x,y))$ along the punctured disc $\Disc^*$
commuting with $\F$, i.e. $\hat\Phi^*\omega\wedge\omega=0$, extends as a  transversely formal diffeomorphism
along the whole disc $\Disc$.
\end{lemma}

\begin{proof} By Proposition \ref{prop:PoincareDulacTransForm}, we can assume that $\F$
is in formal normal form (\ref{eq:PoincareDulacSaddle}) and consider first the linear case: 
we can equivalently define $\F$ with $\omega=\frac{dy}{y}-\lambda\frac{dx}{x}$. 
If we write $\hat\Phi(x,y)=(x,y\cdot g(x,y))$, then the condition $\hat\Phi^*\omega\wedge\omega=0$
writes 
$$dg\wedge\left(\frac{dy}{y}-\lambda\frac{dx}{x}\right)=0\ \ \ \text{(or}\ (x\partial_x+\lambda y\partial_y)\cdot g=0\text{)}$$
i.e. $g$ is a first integral for $\F$. If we decompose $g$ in Laurent power series, then we get
$$g=\sum_{m\in\Z,\ n\ge0}a_{m,n}x^my^n\ \ \ \leadsto\ \ \ 
(x\partial_x+\lambda y\partial_y)\cdot g=\sum_{m\in\Z,\ n\ge0}(m+\lambda n)a_{m,n}x^my^n=0.$$
The obstruction to extend at $0$ comes from non zero coefficients $a_{m,n}$ with negative $m$;
but this can only occur when $\lambda\in\Q_{>0}$, what we have excluded.

Consider now the resonant case, and assume $\F$ is defined by 
$$\omega=\frac{du}{u^{k+1}}+\alpha\frac{du}{u}-\frac{dx}{x}=\left(\frac{1}{u^{k}}+\alpha\right)\frac{du}{u}-\frac{dx}{x}$$
(see proof of Proposition \ref{prop:PoincareDulacTransForm}).
Then we have 
$$\hat\Phi^*\omega=\left(\frac{1}{u^{k}g^k}+\alpha\right)\left(\frac{du}{u}+\frac{dg}{g}\right)-\frac{dx}{x}.$$
The ratio of the closed one-forms $\omega$ and $\hat\Phi^*\omega$ must be a first integral of $\F$
and therefore be constant. Moreover, since they have the same residu along $\Disc^*$ (mind that $g$
is not vanishing), then they should be equal, and therefore $dg\wedge du=0$, i.e. $g=g(u)$.
But non zero coefficients $a_{m,n}$ of $g$ occur only for $n\ge0$, and therefore only for $m\ge0$ since
$u=x^py^q$. So $g$ extends on the whole of $\Disc$.
\end{proof}

\begin{remark}\label{rem:MMsaddlenode}
Our proof of Mattei-Moussu's Theorem works not only for saddles $\lambda<0$,
but also for saddle-node $\lambda=0$ provided that the central manifold is convergent
(i.e. the formal invariant curve along the $0$-eigendirection).
It was proved in \cite{MartinetRamis} that the holonomy of the strong manifold characterize
the foliation, but the proof was completely indirect, by comparing the two moduli spaces.
The original approach of Mattei and Moussu fails to conclude in that case.
Here, we get a direct and simple proof when the central manifold is convergent. 
We will explain later how to adapt to the case of a divergent central manifold (the generic case).
\end{remark}

\begin{remark}\label{rem:MMsaddlenodeCentral}
In the saddle-node case $\lambda=0$, it is well-known that 
the holonomy of the central manifold, when it is convergent, fails to
characterize the saddle-node, even formally. In fact, if we want to normalize
a saddle-node of the form $x^2dy-(1+o(y))ydx$ by a transversly formal change of coordinates
$\Phi(x,y)=(x,\phi(x,y))$ along the disc $\Disc$ or even along the punctured disc $\Disc^*$, 
we find obstructions at each step. Indeed, if it takes the form 
$$\frac{dy}{y}-(1+f(x)y^n+o(y^n)\frac{dx}{x}$$
and we apply a change of the form $y\mapsto y(1+g(x)y^n)$, then we find
$$\frac{dy}{y}-(1+\tilde f(x)y^n+o(y^n)\frac{dx}{x}\ \ \ \text{with}\ \ \ \tilde f=f-x^2g'(x)-ng(x).$$
If we want to make $\tilde f=0$ as it should be for $n>>0$, we see by integration that we 
must set $g=e^{-n/x}\int \frac{e^{n/x}f(x)}{x^2}$ which is multiform on $\Disc^*$ in general.
Moreover, there are many transversely formal symmetries along $\Disc^*$ that do not extend:
we can take any transformation of the form $y\mapsto y g(ye^{1/x})$ with $g(0)\not=0$.
\end{remark}

\begin{remark}\label{rem:CommentMM}
If we want to compare our proof to the original one of Mattei and Moussu,
the present one has the disadvantage that we have to discuss between 
the different formal types, and deal with existence of first integrals, 
closed one-forms; only in the non resonant part $\lambda\not\in\Q$ it is 
very short and easy. On the other hand, this approach can be used to classify
pairs of singular foliations as we shall see in the next section. 
\end{remark}

\begin{remark}\label{rem:MMpairs}
Mattei-Moussu's Theorem can be considered as a classification of pairs of foliations.
Indeed, consider a pair $(\F_1,\F_2)$ of foliations germs such that $\F_1$ is regular and 
$\F_2$ is a saddle, one invariant curve of which is a leaf of $\F_1$. If we have two such pairs,
$(\F_1,\F_2)$ and $(\G_1,\G_2)$, one can easily find coordinates in which 
$\F_1,\G_1$ are vertical, defined by $\ker(dx)$, and the two invariant curves 
of $\F_2,\G_2$ are on coordinate axis. Then, the two pairs are conjugated if and only if
they share the same $\lambda$ and have conjugated holonomies. Indeed,
the conjugacy between the foliations $\F_2$ and $\G_2$ provided by Theorem \ref{thm:MM}
preserves the variable $x$ and therefore $\F_1=\G_1$. This was one of the motivation to us 
for looking at classifications of pairs in a more general setting.
\end{remark}

\section{Reduction of singularities for a pair of foliations}\label{sec:Seidenberg}

Consider a pair $(\F_1,\F_2)$ of holomorphic singular foliations on a complex surface $M$.
Denote by $\Tang(\F_1,\F_2)$ the tangency divisor between the two foliations:
locally, if $\F_i=\ker(\omega_i)$ where $\omega_i$ is a holomorphic one-form with isolated zeroes,
then $\Tang(\F_1,\F_2)$ is defined by the ideal $(f)$ given by $\omega_1\wedge\omega_2=f dx\wedge dy$.
We note that $\Tang(\F_1,\F_2)$ is passing through all singular points of each $\F_i$. Recall Seidenberg's Theorem:

\begin{theorem}[Seidenberg \cite{Seidenberg}]
Let $\F$ be a holomorphic singular foliations on a complex surface $M$.
Then there is a proper map $\pi:\tilde M\to M$ obtained by a finite sequence of punctual blowing-up 
such that the lifted foliation $\tilde\F=\pi^*\F$
has only {\bf reduced singular points}, i.e. of the type
$$\F=\ker(xdy-(\lambda+o(1))ydx)\ \ \ \text{with}\ \ \ \lambda\in\C\setminus\Q_{>0}.$$
Moreover, this property is stable under additional blowing-up.
\end{theorem}

This allows to reduce the study of singular points to non degenerate models. We note that after reduction
of singularities, all singular points have exactly two invariant curves (one of them might be formal divergent
in the saddle-node case $\lambda=0$) which are smooth and transversal.
We want to prove now a similar result for pairs of foliations.

\begin{theorem}\label{thm:ReductionPair}
Let $(\F_1,\F_2)$ be a pair of holomorphic singular foliations on a compact complex surface $M$,
and $\T$ be the support of $\Tang(\F_1,\F_2)$.
Then there is a proper map $\pi:\tilde M\to M$ obtained by a finite sequence of punctual blowing-up 
such that the pair of lifted foliations $(\tilde\F_1,\tilde\F_2)=\pi^*(\F_1\F_2)$ is, at each point of $\tilde M$,
one of the following types: 
\begin{enumerate}
\item $\T=\emptyset$, both $\F_i$ are regular and transversal to each other;
\item $\T$ is smooth, both $\F_i$ are regular, transversal to $\T$;
\item $\T$ is smooth, both $\F_i$ are regular, tangent to $\T$;
\item $\T$ has a normal crossing, both $\F_i$ are regular, tangent to one component of $\T$, 
and transversal to the other one;
\item $\T$ is smooth, $\F_1$ is regular, $\F_2$ has a reduced singular point, 
and $\T$ is a common leaf/invariant curve of $\F_i$;
\item $\T$ has a normal crossing, both $\F_i$ have a reduced singular point, and $\T$ is 
the common set of invariant curves of $\F_i$; moreover, in case both $\F_i$ are saddle-node,
they share the non zero invariant curve.
\end{enumerate}
\end{theorem}

Before proving the theorem, let us comment on the possible types of reduced singular points.
First of all, types (1) - (4) admit simple local normal form:
\begin{enumerate}
\item $(\F_1,\F_2)\sim(dy,dx)$ (the generic regular case);
\item $(\F_1,\F_2)\sim(dy,d(y+x^{k+1}))$ and $\T=(x^k)$, $k\in\Z_{>0}$ (see \cite[Lemma 5.1]{Painleve1et2});
\item $(\F_1,\F_2)\sim(dy,d(y+xy^{k}))$ and $\T=(y^k)$, $k\in\Z_{>0}$ (see \cite[Lemma 5]{AsterisqueRamis});
\item $(\F_1,\F_2)\sim(dy,d(y+x^{k+1}y^{l}))$ and $\T=(x^ky^l)$, $k,l\in\Z_{>0}$. 
\end{enumerate}
Generic points of each branch of $\T$ are either of type (2), or of type (3).

Type (5) splits into 3 cases: \newline
(5.1) $(\F_1,\F_2)\sim(dx,xdy-\lambda ydx)$, $\lambda\in\C\setminus(\R_{\le0}\cup\Q_{>0})$, and $\T=(x)$, or\newline
(5.2) $(\F_1,\F_2)\sim(dx,xdy-(\lambda+o(1))ydx)$, $\lambda\in\R_{\le0}$, and $\T=(x)$, or\newline
(5.3)  $(\F_1,\F_2)\sim(dy,xdy-(1+o(1))y^{k+1}dx)$ and $\T=(y^{k+1})$, $k\in\Z_{>0}$;  \newline
Moreover, the classification of the pair up to analytic diffeomorphism is equivalent to the analytic classification of $\F_2$
in each case.
Poincar\'e Linearization Theorem provides linearization in (5.1). The saddle case (5.2) with $\lambda<0$  
corresponds to Mattei-Moussu's Theorem (see Remark \ref{rem:MMpairs}).
The saddle-node case splits into (5.2)  with $\lambda=0$ (see Remark \ref{rem:MMsaddlenode}) 
and (5.3) (see \cite[Proposition 2.2.3]{MartinetRamis}). Moreover, for fixed $\lambda$,
the analytic class of $\F_2$ is characterized by $\lambda$ and the analytic class of its holonomy 
along the invariant curve $y=0$ in (5.2) (Mattei-Moussu's Theorem), or $x=0$ in (5.3) 
(see \cite[Corollaire 3.3]{MartinetRamis}).

So far, only the classification of type (6) was not known.
In this direction, the first author recently obtained (see \cite{TheseArame,PapierPairFol})
the complete classification in the case $\Tang(\F_1,\F_2)$ is reduced 
(i.e. $\T=(xy)$, without multiplicity). We will state the results in Section \ref{Sec:Pairs}.
We note that in case (6), any saddle occuring among $\F_i$ must have analytic
invariant curves, i.e. not formal divergent component, since they are contained in $\T$ 
which is analytic.

The list of Theorem \ref{thm:ReductionPair} is not stable by arbitrary blowing-up:
if we blow-up items (2) or (3) for instance, then we have to blow-up more in order to recover 
a pair with reduced singular points as in the statement. 

\begin{proof}Let $M$, $(\F_1,\F_2)$ and $\T$ as in the statement.
First of all, there exists a proper map $\pi:\tilde M\to M$ obtained by a finite sequence of punctual blowing-up
such that the pull-back $\pi^*\F_i$ have both only reduced singular points. This is done by applying twice 
Seidenberg's Theorem \cite{Seidenberg}, successively for $\F_1$ and $\F_2$. 
These properties are stable by additional blowing-up
and the final map $\pi:\tilde M\to M$ of the statement will be obtained by additional blowing-ups. 
Therefore, we can assume without lost of generality that both $\F_i$ have reduced singular points from the beginning.

Recall that $\T$ is passing through all singular points of $\F_1$ or $\F_2$.
Outside the support of $\T$, the pair is regular, of type (1). Along each irreducible component
of $\T$, the pair is generically of type (2) or (3), depending whether that component is generically transversal
to the $\F_i$'s, or is $\F_i$-invariant; moreover, this generic feature is for a Zariski open set of $\T$. 
We therefore conclude that, apart from a finite set of points in $M$, the situation 
is as in (1), (2) and (3) of the list. So the problem is local: let us consider a point $p$
where $(\F_1,\F_2,\T)$ is not as (1), (2) or (3), and prove that after finitely many 
blowing-up infinitesimally close to $p$, we get only points in the list (1)-(6);
as we shall see, we will get only local models of type (3)-(6) along the exceptional divisor.

In order to show this, let us  consider the local formal invariant curve $\Gamma_i$ of $\F_i$ at $p$:
it consists of one or two smooth and transversal branches depending on whether 
the foliation is regular or singular. Then consider the germ $\Gamma=\Gamma_1\cup\Gamma_2\cup\T$ at $p$;
mind that they can share common branches.
Then we can blow-up until $\tilde\Gamma=\pi^*\Gamma$ has only normal crossing singular points.
We note that all local invariant curves for $\tilde\F_i=\pi^*\F_i$ are contained in $\tilde\Gamma$;
otherwise, any extra invariant curve would descend as an additional invariant curve for $\F_i$
outside $\Gamma$, providing a contradiction. We also note that $\tilde\Gamma$ contains 
the support of $\Tang(\tilde\F_1,\tilde\F_2)$ for a similar reason.
We claim that all points along $\tilde\Gamma$ are of type (3)-(6). Let us check this.

Let us consider first the case where $\tilde\Gamma$ is smooth. Then the two foliations are smooth
(only one invariant curve) with $\tilde\Gamma$ as a common leaf and no other tangency: 
we are in type (3). Now, in order to consider the case $\tilde\Gamma$ has two components,
note that any common component between two among the three curves 
$\tilde\Gamma_i=\pi^*\Gamma_i$, $i=1,2$, and $\tilde\T=\pi^*\T$ 
must be also a component of the third curve. Indeed, if $\tilde\Gamma_1$ and $\tilde\Gamma_2$
have a common component, then $\tilde\F_1$ and $\tilde\F_2$ must be tangent along this branch and 
it must be a component of $\tilde\T$; in a similar way, if $\T$ and $\tilde\Gamma_1$, say, have a common 
component, then that branch of $\tilde\Gamma_1$ must be invariant by $\tilde\F_2$ as well, and be a component
of $\tilde\Gamma_2$. We conclude that, each branch of $\tilde\Gamma$ consists of 
\begin{itemize}
\item either a branch of $\tilde\T$ only, 
\item or a branch of $\tilde\Gamma_i$ only, 
\item or a common branch of $\tilde\Gamma_1$, $\tilde\Gamma_2$ and $\tilde\T$.
\end{itemize}
In the first case above, then $\tilde\F_1$ and $\tilde\F_2$ are regular (only one invariant curve contined in the other branch of $\tilde\Gamma$) 
and transversal to that branch of $\tilde\T$ only:
we are in case (4). In the second case above, say $i=1$, then $\tilde\F_2$ is regular (only one invariant curve)
and transversal to that branch of $\tilde\Gamma_1$, and tangent to the other branch: we are in case (5).
Finally, if all branches are shared by $\tilde\T$, $\tilde\Gamma_1$ and $\tilde\Gamma_2$, then we are in case (6).
Finally, in case (6),  if ever $\tilde\F_1$ and $\tilde\F_2$ are saddle-nodes oriented in opposite way, i.e. each branch of $\T$ 
is the strong manifold (non zero eigendirection) of one and the central manifold (zero eigendirection)
of the other, then one check that after one additional blowing-up, we get two singular pairs along the exceptional
divisor with only one saddle-node in each pair.

To conclude, applying the above strategy of resolution by blowing-up at each point $p$ in $M$ where 
the pair is not locally of type (1)-(3) yields a global proper map $\pi:\tilde M\to M$ such that 
the lifted pair is locally of type (1)-(6) everywhere.
\end{proof}

\section{Reduced pairs of foliations: a partial classification}\label{Sec:Pairs}

This section is devoted to announcement of results of the first author in her thesis
\cite{TheseArame}. All details will appear in a forthcoming paper \cite{PapierPairFol}.

We consider a local pair $(\F_1,\F_2)$ of type (6) in the list of Theorem \ref{thm:ReductionPair}
and assume that $\Tang(\F_1,\F_2)=(xy)$ is reduced, i.e. without multiplicity.
One easily check that this implies that we can write 
$$\F_i=\ker(\ xdy-(\lambda_i+o(1))ydx)$$
with $\lambda_i\in\C\setminus\Q_{\ge0}$, and with $\lambda_1\not=\lambda_2$ 
(otherwise $\Tang(\F_1,\F_2)$ cannot be reduced). Conversely, any pair $(\F_1,\F_2)$
as above with $\lambda_1\not=\lambda_2$ is of type (6) with reduced tangency along axis.

Denote 
$$\Diff^\times(\C^2,0)=\{\Phi(x,y)=(x a(x,y),y b(x,y))\ ;\ a,b\in\OO^\times(\C^2,0)\}\subset\Diff(\C^2,0)$$
the subgroup of those diffeomorphisms preserving the axis (without permutation) and
$$\F\sim\G\ \ \ \text{or}\ \ \ (\F_1,\F_2)\sim(\G_1,\G_2)$$
when two foliations or two pairs are conjugated by a diffeomorphism $\Phi\in\Diff^\times(\C^2,0)$.

\begin{theorem}[A. A. Diaw \cite{TheseArame,PapierPairFol}]\label{thm:Pair}
 Consider two pairs $(\F_1,\F_2)$ and $(\G_1,\G_2)$ of the form
$$\F_i=\ker(xdy-(\lambda_i+f_i(x,y))ydx)\ \ \ \text{and}\ \ \ \G_i=\ker( xdy-(\lambda_i+g_i(x,y))ydx)$$
with $\lambda_1\not=\lambda_2$, $\lambda_i\in\C\setminus\Q_{\ge 0}$, and $f_i,g_i\in\OO(\C,0)$ vanishing at $0$.
Then we have 
$$(\F_1,\G_1)\sim(\F_2,\G_2)\ \ \ \Leftrightarrow\ \ \ \F_1\sim\F_2\ \text{and}\ \G_1\sim\G_2.$$
\end{theorem}

In other words, if there are $\Phi_i\in\Diff^\times(\C^2,0)$ such that $\F_i=\Phi_i^*\G_i$ for $i=1,2$,
then there exists a single $\Phi\in\Diff^\times(\C^2,0)$ such that $\F_i=\Phi^*\G_i$.
An equivalent statement is given by

\begin{theorem}[A. A. Diaw \cite{TheseArame,PapierPairFol}]\label{thm:DecomposeDiffeo}
Consider a pair $(\F_1,\F_2)$ as in Theorem \ref{thm:Pair}. 
Then, for any $\Psi\in\Diff^\times(\C^2,0)$, there exist $\Psi_1,\Psi_2\in\Diff^\times(\C^2,0)$
such that 
$$\Psi=\Psi_1\circ\Psi_2\ \ \ \text{and}\ \ \ \Psi_1^*\F_i=\F_i,\ i=1,2.$$
\end{theorem}

\begin{proof}[Proof of the equivalence between the two statements]
Assume we are two pairs $(\F_1,\F_2)$ and $(\G_1,\G_2)$ as in Theorem \ref{thm:Pair}
and $\Phi_i\in\Diff^\times(\C^2,0)$ such that $\F_i=\Phi_i^*\G_i$. Then we can apply
Theorem \ref{thm:DecomposeDiffeo} to decompose $\Psi:=\Phi_1^{-1}\circ\Phi_2=\Psi_1\circ\Psi_2$
and check that $\Phi:=\Phi_1\circ\Psi_1=\Phi_2\circ\Psi_2^{-1}$ is conjugating
the pairs $(\F_1,\F_2)$ to $(\G_1,\G_2)$. Conversely, given $\Psi\in\Diff^\times(\C^2,0)$,
we can apply Theorem \ref{thm:Pair} to produce conjugacy of pairs
$$(\F_1,\Psi^*\F_2)\stackrel{\Phi_1}{\longrightarrow}(\F_1,\F_2)$$
so that we have a conjugacy
$$(\F_1,\F_2)\stackrel{\Psi\circ\Phi_1^{-1}}{\longrightarrow}(\Psi_*\F_1,\F_2).$$
Then we have decomposition $\Psi=\Psi_1\circ\Psi_2$ where
$\Psi_1=\Phi_1$ preserves $\F_1$, and $\Psi_2=\Psi\circ\Phi_1^{-1}$ preserves $\F_2$.
\end{proof}

The technics involved to prove Theorem \ref{thm:Pair} use local study of 
a pair of regular foliations along a common leaf, as was already done in 
\cite{AsterisqueRamis,OlivierBras,FredOlivier,OlivierGenre2}; this allows to construct conjugacy 
between of pairs of foliations at the neighborhood of a punctured disc $\Disc^*$ contained in the 
invariant curve.
On the other hand, one can prove a version of the decomposition in Theorem \ref{thm:DecomposeDiffeo}
in the  transversely formal setting along the complete disc $\Disc$, implying a transversely formal conjugacy 
between the pairs of foliations. 
Then a careful study of transversely formal symmetries of a pair along $\Disc^*$ allow the first author to 
conclude to the extension of the conjugacy at $0$, likely as in our proof of Mattei-Moussu's Theorem.

We expect that it will be possible to obtain the complete classification even for non reduced
tangency divisors, but there will be additional invariants in that case, and more complicated 
statement. The reduced case is motivated by the study of Hilbert Modular Surfaces which
admit a pair of foliations with reduced tangency divisor (see \cite{ErwanFred}).

\end{document}